\title{The matroid of a graphing}
\author{László Lovász\footnote{Alfr\'ed R\'enyi Institute of Mathematics. Research supported by ERC Synergy Grant
No.~810115.}}

\documentclass[12pt,fleqn]{article}
\usepackage{amsmath,amssymb,graphicx,bbm,bm,delarray,pict2e,ntheorem}
\usepackage{hyperref}
\usepackage[latin2]{inputenc}
\usepackage[normalem]{ulem}

\newtheorem{theorem}{Theorem}[section]

\newtheorem{lemma}[theorem]{Lemma}

\newtheorem{corollary}[theorem]{Corollary}

\theorembodyfont{\rmfamily}
\newtheorem{remark}[theorem]{Remark}
\newtheorem{example}[theorem]{Example}

\newtheorem{exc}[theorem]{Exercise}

\newenvironment{proof}{\medskip\noindent{\bf Proof. }}{\hfill$\square$\medskip}
\newenvironment{proof*}[1]{\medskip\noindent{\bf Proof of #1.}}{\hfill$\square$\medskip}

\begin{document}

\addtolength{\baselineskip}{3pt}

\def\AA{\mathcal{A}}\def\BB{\mathcal{B}}\def\CC{\mathcal{C}}
\def\DD{\mathcal{D}}\def\EE{\mathcal{E}}\def\FF{\mathcal{F}}
\def\GG{\mathcal{G}}\def\HH{\mathcal{H}}\def\II{\mathcal{I}}
\def\JJ{\mathcal{J}}\def\KK{\mathcal{K}}\def\LL{\mathcal{L}}
\def\MM{\mathcal{M}}\def\NN{\mathcal{N}}\def\OO{\mathcal{O}}
\def\PP{\mathcal{P}}\def\QQ{\mathcal{Q}}\def\RR{\mathcal{R}}
\def\SS{\mathcal{S}}\def\TT{\mathcal{T}}\def\UU{\mathcal{U}}
\def\VV{\mathcal{V}}\def\WW{\mathcal{W}}\def\XX{\mathcal{X}}
\def\YY{\mathcal{Y}}\def\ZZ{\mathcal{Z}}

\def\Ab{\mathbf{A}}\def\Bb{\mathbf{B}}\def\Cb{\mathbf{C}}
\def\Db{\mathbf{D}}\def\Eb{\mathbf{E}}\def\Fb{\mathbf{F}}
\def\Gb{\mathbf{G}}\def\Hb{\mathbf{H}}\def\Ib{\mathbf{I}}
\def\Jb{\mathbf{J}}\def\Kb{\mathbf{K}}\def\Lb{\mathbf{L}}
\def\Mb{\mathbf{M}}\def\Nb{\mathbf{N}}\def\Ob{\mathbf{O}}
\def\Pb{\mathbf{P}}\def\Qb{\mathbf{Q}}\def\Rb{\mathbf{R}}
\def\Sb{\mathbf{S}}\def\Tb{\mathbf{T}}\def\Ub{\mathbf{U}}
\def\Vb{\mathbf{V}}\def\Wb{\mathbf{W}}\def\Xb{\mathbf{X}}
\def\Yb{\mathbf{Y}}\def\Zb{\mathbf{Z}}

\def\ab{\mathbf{a}}\def\bb{\mathbf{b}}\def\cb{\mathbf{c}}
\def\db{\mathbf{d}}\def\eb{\mathbf{e}}\def\fb{\mathbf{f}}
\def\gb{\mathbf{g}}\def\hb{\mathbf{h}}\def\ib{\mathbf{i}}
\def\jb{\mathbf{j}}\def\kb{\mathbf{k}}\def\lb{\mathbf{l}}
\def\mb{\mathbf{m}}\def\nb{\mathbf{n}}\def\ob{\mathbf{o}}
\def\pb{\mathbf{p}}\def\qb{\mathbf{q}}\def\rb{\mathbf{r}}
\def\sb{\mathbf{s}}\def\tb{\mathbf{t}}\def\ub{\mathbf{u}}
\def\vb{\mathbf{v}}\def\wb{\mathbf{w}}\def\xb{\mathbf{x}}
\def\yb{\mathbf{y}}\def\zb{\mathbf{z}}

\def\Abb{\mathbb{A}}\def\Bbb{\mathbb{B}}\def\Cbb{\mathbb{C}}
\def\Dbb{\mathbb{D}}\def\Ebb{\mathbb{E}}\def\Fbb{\mathbb{F}}
\def\Gbb{\mathbb{G}}\def\Hbb{\mathbb{H}}\def\Ibb{\mathbb{I}}
\def\Jbb{\mathbb{J}}\def\Kbb{\mathbb{K}}\def\Lbb{\mathbb{L}}
\def\Mbb{\mathbb{M}}\def\Nbb{\mathbb{N}}\def\Obb{\mathbb{O}}
\def\Pbb{\mathbb{P}}\def\Qbb{\mathbb{Q}}\def\Rbb{\mathbb{R}}
\def\Sbb{\mathbb{S}}\def\Tbb{\mathbb{T}}\def\Ubb{\mathbb{U}}
\def\Vbb{\mathbb{V}}\def\Wbb{\mathbb{W}}\def\Xbb{\mathbb{X}}
\def\Ybb{\mathbb{Y}}\def\Zbb{\mathbb{Z}}

\def\Af{\mathfrak{A}}\def\Bf{\mathfrak{B}}\def\Cf{\mathfrak{C}}
\def\Df{\mathfrak{D}}\def\Ef{\mathfrak{E}}\def\Ff{\mathfrak{F}}
\def\Gf{\mathfrak{G}}\def\Hf{\mathfrak{H}}\def\If{\mathfrak{I}}
\def\Jf{\mathfrak{J}}\def\Kf{\mathfrak{K}}\def\Lf{\mathfrak{L}}
\def\Mf{\mathfrak{M}}\def\Nf{\mathfrak{N}}\def\Of{\mathfrak{O}}
\def\Pf{\mathfrak{P}}\def\Qf{\mathfrak{Q}}\def\Rf{\mathfrak{R}}
\def\Sf{\mathfrak{S}}\def\Tf{\mathfrak{T}}\def\Uf{\mathfrak{U}}
\def\Vf{\mathfrak{V}}\def\Wf{\mathfrak{W}}\def\Xf{\mathfrak{X}}
\def\Yf{\mathfrak{Y}}\def\Zf{\mathfrak{Z}}

\def\alphab{{\boldsymbol\alpha}}\def\betab{{\boldsymbol\beta}}
\def\gammab{{\boldsymbol\gamma}}\def\deltab{{\boldsymbol\delta}}
\def\etab{{\boldsymbol\eta}}\def\zetab{{\boldsymbol\zeta}}
\def\kappab{{\boldsymbol\kappa}}
\def\lambdab{{\boldsymbol\lambda}}\def\mub{{\boldsymbol\mu}}
\def\nub{{\boldsymbol\nu}}\def\pib{{\boldsymbol\pi}}
\def\rhob{{\boldsymbol\rho}}\def\sigmab{{\boldsymbol\sigma}}
\def\taub{{\boldsymbol\tau}}\def\epsb{{\boldsymbol\varepsilon}}
\def\phib{{\boldsymbol\varphi}}\def\psib{{\boldsymbol\psi}}
\def\xib{{\boldsymbol\xi}}\def\omegab{{\boldsymbol\omega}}
\def\intl{\int\limits}
\def\sqprod{\mathbin{\square}}

\def\ybb{\mathbbm{y}}
\def\one{{\mathbbm1}}
\def\two{{\mathbbm2}}
\def\R{\Rbb}\def\Q{\Qbb}\def\Z{\Zbb}\def\N{\Nbb}\def\C{\Cbb}
\def\wh{\widehat}
\def\wt{\widetilde}

\def\eps{\varepsilon}
\def\sgn{{\rm sign}}
\def\dd{{\sf d}}
\def\Rv{\overleftarrow}
\def\Pr{{\sf P}}
\def\E{{\sf E}}
\def\T{^{\sf T}}
\def\proofend{\hfill$\square$}
\def\id{\hbox{\rm id}}
\def\conv{\hbox{\rm conc}}
\def\lin{\hbox{\rm lib}}
\def\conv{\hbox{\rm conc}}
\def\Dim{\hbox{\rm Dim}}
\def\const{\hbox{\rm const}}
\def\vol{\text{\rm vol}}
\def\diam{\text{\rm diam}}
\def\corank{\hbox{\rm cork}}
\def\cork{\hbox{\rm cork}}
\def\OR{\mathcal{OR}}
\def\GOR{\mathcal{GOB}}
\def\NOR{\mathcal{NOR}}
\def\LGOR{\mathcal{ALGOR}}
\def\cro{\text{\rm CR}}
\def\supp{\text{\rm sup}}
\def\grad{\text{\rm grad}}
\def\rk{\overline{\rho}}
\def\srk{\hbox{\rm src}}
\def\diag{{\rm diag}}
\def\pw{{\sf w}_\text{\rm prod}}
\def\tw{{\sf w}_\text{\rm tree}}
\def\aw{{\sf w}_\text{\rm alb}}
\def\bw{{\sf bow}}
\def\ld{{\sf d}_{\rm loc}}
\def\hd{{\sf d}_{\rm har}}
\def\tv{\text{\rm tv}}
\def\Tr{\hbox{\rm Tar}}
\def\tr{\hbox{\rm tr}}
\def\Prob{\hbox{\rm Pr}}
\def\bl{\text{{\rm bl}}}
\def\abl{\hbox{{\rm abl}}}
\def\Id{\hbox{\rm Id}}
\def\aff{\text{\rm ra}}
\def\MC{\CC_{\max}}
\def\Inf{\text{\sf Inf}}
\def\Str{\text{\sf Str}}
\def\Rig{\text{\sf Rig}}
\def\Mat{\text{\sf Mat}}
\def\comm{{\sf comm}}
\def\maxcut{{\sf maxcut}}
\def\disc{\text{\sf disc}}
\def\cond{\Phi}
\def\val{\text{\sf val}}
\def\dist{d_{\rm qu}}
\def\dhaus{d_{\rm haus}}
\def\dlp{d_{\rm LP}}
\def\dact{d_{\rm act}}
\def\Ker{{\rm Ker}}
\def\Rng{{\rm Rng}}
\def\gdim{{\rm gdim}}
\def\gap{\text{\rm gap}}
\def\intl{\int\limits}
\def\et{\qquad\text{and}\qquad}
\def\fin{\text{\sf fin}}
\def\Bd{{\sf Bd}}
\def\ba{{\sf ba}}
\def\ca{{\sf ca}}
\def\matp{{\sf mm}}
\def\sep{{\sf sep}}
\def\bp{{\sf bp}}
\def\fg{\varphi}
\def\fgx{{\varphi}^\sqcap}
\def\lc{\bullet}
\def\li{{\rm li}}
\def\ui{{\rm ui}}
\def\ls{{\rm ls}}
\def\us{{\rm us}}
\long\def\ignore#1{}

\def\QR{{R^-}}

\def\url{}
\maketitle

{\small \tableofcontents}

\begin{abstract}
Graphings serve as limit objects for bounded-degree graphs. We define the
``cycle matroid'' of a graphing as a submodular setfunction, with values in
$[0,1]$, which generalizes (up to normalization) the cycle matroid of finite
graphs. We prove that for a Benjamini--Schramm convergent sequence of graphs,
the total rank, normalized by the number of nodes, converges to the total rank
of the limit graphing.
\end{abstract}

\section{Introduction}

Graphings serve as limit objects for bounded-degree graphs \cite{BSch,HomBook}.
They can be defined as Borel graphs on a Borel probability space with a certain
measure-preservation property. Our goal is to extend the definition of the
cycle matroid of a graph to graphings.

In the cycle matroid of a finite graph $G=(V,E)$, the rank of a set $X\subseteq
E$ is determined by the partition $\PP$ of the subgraph $(V,X)$ into connected
components: we have $r(X)=|V|-|\PP|$. This rank function is submodular:
\begin{equation}\label{EQ:SUBMOD}
r(X\cup Y)+r(X\cap Y)\le r(X)+r(Y)\qquad (X,Y\subseteq E).
\end{equation}
Let us also remark that a set $X$ is independent in this matroid if and only if
the rank function $r$, restricted to $X$, is just the cardinality function.

One could introduce the rank function on the edge-set of a graphing by defining
it for finite subsets $X$ of the edge set as in the cycle matroid of the graph
formed by $X$ (ignoring isolated nodes). One could extend the rank function to
infinite sets of edges by giving it the value $\infty$. We choose a perhaps
more interesting approach, generalizing the {\it normalized rank function},
defined for a finite graph $G=(V,E)$ as $\rho(X) = r(X)/|V|$. Clearly this
function is submodular, with values in $[0,1]$.

The idea of the generalization is the following observation: Let $\PP$ be the
partition of the node set $V$ into the connected components of subgraph
$(V,X)$. For $v\in V$, we denote by $\PP_v$ the partition class containing $v$.
Let $\ub$ be a uniform random node of $V$, then
\begin{equation}\label{EQ:PART_NO}
\E\Big(\frac1{|\PP_\ub|}\Big) = \frac{|\PP|}{|V|} = 1-\rho(X).
\end{equation}

This definition can be extended almost verbatim to partitions of graphings into
connected component; however, measurability issues and infinite classes make it
technically involved to prove submodularity and other properties. To define our
generalization to the infinite case, we consider a standard Borel sigma-algebra
$(J,\BB)$, and ``measurable'' partitions of $J$ into sets in $\BB$, leading to
a setfunction defined on these partitions, which is ``submodular'' in some
sense.

We will be interested in partitions with many finite classes: If all, or almost
all, partition classes are infinite, then the expectation above will be $0$ and
$\rho(X)=1$.

We define the ``cycle matroid'' of a graphing as a submodular setfunction
defined on Borel subsets of the edge set $E$, with values in $[0,1]$. This
generalizes the (normalized) cycle matroid of finite graphs. The proof of
submodularity is obtained through a more general result about a function
defined on measurable partitions of a Borel space.

What are the ``bases'' of this generalized matroid? The results of
\cite{Lov23a} suggest that these bases could be defined as finitely additive
measures (charges) $\alpha$ on Borel subsets of edges, upper bounded by $\rho$,
with $\alpha(E)=\rho(E)$. We show that certain spanning forests of the graphing
do give rise to such charges, but we don't have a full characterization.

A next step should be to develop the limit theory for $\rho$. There have been
attempts to define convergence of matroids, at least for special families
\cite{Bj87,Hai,BjLo,KKLM}, but these don't seem to apply to our case.

A sequence of finite graphs with all degrees bounded by $D$ is said to {\it
locally converge} to a graphing $\Gb$, if for every $r\ge 1$, the distribution
of the $r$-ball centered at a random node of $G_n$ converges to the
distribution of the $r$-ball centered at a random point of $\Gb$. (The notion
of convergence of a bounded-degree graph sequence is due to Benjamini and
Schramm \cite{BSch}, with a different construction for the limit object; see
\cite{HomBook} for convergence to graphings.) To develop a limit theory for
cycle matroids, we want to show that if a sequence of bounded-degree graphs
$G_n$ converges to a graphing $\Gb$, then the setfunctions $\rho_{G_n}$
converge to the setfunction $\rho_\Gb$ in some sense. This is not resolved in
this paper; but we do prove that for a Benjamini--Schramm convergent sequence
of graphs, the rank of the edge set, normalized by the number of nodes,
converges to the rank of the limit graphing.

\section{$\BB$-partitions}

\subsection{The semilattice of $\BB$-partitions}

Let $(J,\BB)$ be a standard Borel space, let $\PP$ be a partition of $J$, and
for $X\subseteq J$, define
\[
\PP(X)=\cup\{P\in\PP:~P\cap X\not=\emptyset\}.
\]
We say that $\PP$ is a $\BB$-partition, if $\PP(A)\in\BB$ whenever $A\in\BB$.
For a partition $\PP$ we denote the class containing $x\in J$ by $\PP_x$.
Clearly $\PP_x\in\BB$ if $\PP$ is a $\BB$-partition. We denote by $\PP_\fin$
the family of finite subsets of $\PP$.

If $\PP$ and $\QQ$ are partitions of $J$, we denote by $\PP\land\QQ$ and
$\PP\lor\QQ$ their meet and join in the partition lattice.

\begin{lemma}\label{LEM:JOIN-B}
{\rm(a)} If $\PP$ and $\QQ$ are $\BB$-partitions, then so is $\PP\lor\QQ$.
{\rm(b)} If, in addition, $\QQ$ has a countable number of classes, then
$\PP\land\QQ$ is a $\BB$-partition.
\end{lemma}

\begin{proof}
(a) Let $A\in\BB$, and define $U_0=A$, and recursively
\[
U_{k+1} =
  \begin{cases}
    \PP(U_k), & \text{if $k$ is even}, \\
    \QQ(U_k), & \text{if $k$ is odd}.
  \end{cases}
\]
Then $U_k\in\BB$ for all $k$, and so $(\PP\lor\QQ)(A) =\cup_k U_k\in\BB$.

\medskip

(b) Let $U=\bigcup_{Q\in\QQ} Q\cap\PP(Q\cap A)$. We claim that
\begin{equation}\label{EQ:PPQQB}
(\PP\land\QQ)(A)=\{x\in J:~\PP_x\cap \QQ_x\cap A\not=\emptyset\}= U.
\end{equation}
The first equation is trivial. To verify the second, assume that $x\in J$ has
the property that $\PP_x\cap \QQ_x\cap A\not=\emptyset$. Then there is a $y\in
\PP_x\cap \QQ_x\cap A$. Hence $y\in\PP_x$, which implies that $\PP_x=\PP_y$.
Furthermore, $y\in\QQ_x\cap A$ implies that $\PP_y\subseteq\PP(\QQ_x\cap A)$,
and hence $x\in\PP(\QQ_x\cap A)$. Since trivially $x\in\QQ_x$, we have $x\in
\QQ_x\cap\PP(\QQ_x\cap A)$, so $x\in U$. Conversely, if $x\in Q\cap\PP(Q\cap
A)$ for some $Q$, then clearly $Q=Q_x$, and $x\in\PP(\QQ_x\cap A)$ means that
$\PP_x$ intersects $\QQ_x\cap A$, so $\PP_x\cap \QQ_x\cap A\not=\emptyset$.

Now $Q\cap\PP(Q\cap A)$ is clearly a Borel set, and since $\QQ$ has a countable
number of classes, \eqref{EQ:PPQQB} implies that $U=(\PP\land\QQ)(A)$ is Borel.
\end{proof}

\begin{lemma}\label{LEM:FIN-PART}
Let $\PP$ be a $\BB$-partition, $A\in\BB$, and let $A_k=\cup\{P\in\PP:~|P\cap
A|=k\}$. Then $A_k\in\BB$ for $k=0,1,\ldots,\infty$.
\end{lemma}

This lemma implies that $|\PP_x|$ is a measurable function of $x$ (which may
have infinite values). In particular, the union of $k$-element partition
classes is in $\BB$, and hence the union $\PP_\text{fin}$ of finite classes of
$\PP$ is in $\BB$.

\begin{proof}
We may assume that $(J,\BB)$ is the sigma-algebra of Borel sets in $(0,1)$. Let
\[
C_k=\bigcup_{0=r_0<r_1<\ldots <r_k=1\atop r_i\in\Qbb} \bigcap_{i=0}^{k-1}\PP\big(A\cap(r_i,r_{i+1})\big).
\]
Then $C_k$ is a Borel set, which is the union of partition classes with at
least $k$ elements in $A$. So $A_k=C_k\setminus C_{k+1}$ is also Borel.
\end{proof}

Given a partition $\PP$, let us call a set $S\in\BB$ a {\it finite-class
representative}, if $|S\cap P|=1$ for every finite $P\in\PP$ and $|S\cap P|=0$
for every infinite $P\in\PP$.

\begin{lemma}\label{LEM:REPRESENTATIVE}
Every $\BB$-partition has a finite-class representative.
\end{lemma}

\begin{proof}
Again, assume that $(J,\BB)=(0,1)$, and for every finite class $P\in\PP$, let
$s_P$ be its minimal element. It suffices to show that
\[
S=\{s_P:~P\in\PP, P\text{ finite}\}
\]
is a Borel set. Indeed, the set $U_r=\cup\{P\in\PP:~|P\cap(0,r)|=1\}$ is Borel
by Lemma \ref{LEM:FIN-PART}, and then so is $U_r\cap (\cup\PP_\fin)$. Thus the
set
\[
V_r=U_r\cap (\cup\PP_\fin)\cap(0,r) = \big\{s_P:~P\text{ finite}, P\cap(0,r)=\{s_P\}\big\}.
\]
is Borel, and hence $S=\cup_{r\in\Qbb\cap(0,1)} V_r$ is a Borel set.
\end{proof}

\subsection{Re-randomization}

Now let $\pi$ be a probability measure on $(J,\BB)$, and let $\ub$ be a random
point from $\pi$. If $\PP_\ub$ is finite, then let $\vb$ be a uniform random
point of $P$; else, let $\vb=\ub$. We say that $\vb$ is obtained by {\it
re-randomizing $\ub$ along $\PP$}. We say that $\PP$ has the {\it
re-randomizing property}, if re-randomizing results in a point distributed
according to $\pi$.

Another way to express this property is that if $\mu_x$ is the uniform
distribution on $\PP_x$ if $\PP_x$ is finite, and $\mu_x=\delta_x$ otherwise,
then the mixture of the measurable family $M=(\mu_x:~x\in J)$ of measures by
$\pi$ is $\pi$ again. It is easy to check that every partition of a finite set
endowed with the uniform distribution has the re-randomizing property.

The significance of the re-randomizing property for us is expressed by the
following easy fact:

\begin{lemma}\label{LEM:RRAND-INT}
Let $\PP$ be a $\BB$-partition of $J$ with the re-randomizing property. Let
$f:~J\to\R$ be an integrable function such that $f(x)=0$ if $\PP_x$ is
infinite, and $\sum_{x\in P}f(x)=0$ for every finite class $P\in\PP$. Let $\xb$
be a random point from $\pi$. Then $\E f(\xb)=0$.
\end{lemma}

\begin{proof}
Let $\yb$ be obtained from $\xb$ by re-randomizing along $\PP$. Then $\E
f(\yb)=0$, since this holds for every $\xb$. Since $\yb$ has the same
distribution as $\xb$, the lemma follows.
\end{proof}

Not every $\BB$-partition has the re-randomizing property.

\begin{example}\label{EXA:NOT-RERAND}
Let $(J,\BB,\pi)$ be the $[0,1]$ interval with the Borel sets and the Lebesgue
measure. Partition the interval $[0,1)$ into pairs $\{x,2x+1/3\}$, where $0\le
x<1/3$. Then for a random point $\ub$ from $\pi$, we have
$\Pr(\ub\in[0,1/3))=1/3$, but if we re-randomize to get $\vb$, then
$\Pr(\vb\in[0,1/3))=1/2$.
\end{example}

There are two trivial operations on partitions that preserve the re-randomizing
property: we can merge all infinite classes into one, or resolve all infinite
classes into singletons. Let us state two less trivial constructions.

\begin{lemma}\label{LEM:REFINE-RESAMPLE}
Let $\PP$ and $\QQ$ be $\BB$-partitions of $J$ such that $\QQ$ is obtained from
$\PP$ by splitting some of its finite classes. If $\PP$ has the re-randomizing
property, then so does $\QQ$.
\end{lemma}

\begin{proof}
Let us do the following experiment: let $\ub$ be a random point from $\pi$; let
$\vb$ be obtained from $\ub$ by re-randomizing along $\PP$; then let $\wb$ be
obtained by re-randomizing along $\QQ$. Then (by the re-randomizing property of
finite partitions) $\wb$ can be obtained by re-randomizing $\ub$ along $\PP$,
so it has distribution $\pi$. But it can be obtained by re-randomizing $\vb$
along $\QQ$. Since $\vb$ has distribution $\pi$, this proves the lemma.
\end{proof}

\begin{lemma}\label{LEM:LOR-RESAMPLE}
Let $\PP$ and $\QQ$ be $\BB$-partitions with the re-randomizing property. Then
so is $\PP\lor\QQ$.
\end{lemma}

\begin{proof}
Let us construct a bipartite multigraph $H_Z$ for every finite class $Z$ of
$\PP\lor\QQ$ as follows. Let $V(H_Z)=U\cup W$, where $U=\{P\in\PP:~P\subseteq
Z\}$ and $W=\{Q\in\QQ:~Q\subseteq Z\}$. Let $X\in\PP$ be connected to $Y\in\QQ$
by $|X\cap Y|$ edges. The edges of $H_Z$ can be identified with the elements of
$Z$, and the degree of $X\in V(H_Z)$ is $|X|$. It is also clear that $H_Z$ is
connected by the definition of $\PP\lor\QQ$.

Let $\ub$ be a random point of $J$ from the distribution $\pi$, and let $Z$ be
the partition class of $\PP\lor\QQ$ containing $\ub$. Assume that $Z$ is
finite. Let us do a random walk on $H_Z$ starting at $v^0=\PP_\ub$. The
distribution $\eta_{Z,k}$ of the edge used in step $k$ tends to the uniform
distribution on all edges of $H_Z$, by standard results on random walks on
graphs.

Let $(v^0,v^1,\ldots)$ be the sequence of edges of $H_Z$ in this random walk;
equivalently, $(v^0,v^1,\ldots)$ is a sequence of points in $Z$. It follows by
the re-randomizing property of $\PP$ and $\QQ$ that for every $k$, the
distribution of $v^k$ is just $\pi$. It follows that if we choose a random
point $\ub$ from $\pi$ and re-randomize it along $\PP\lor\QQ$ to get $\vb$,
then $\vb$ has distribution $\pi$, so $\PP\lor\QQ$ has the re-randomizing
property.
\end{proof}

\section{A supermodular function on partitions}\label{SSEC:PART-SUBMOD}

Let $\PP$ be a $\BB$-partition of $J$. Define
\begin{equation}\label{EQ:PSI}
\psi(\PP) = \E\Big(\frac1{|\PP_\ub|}\Big),
\end{equation}
where $\ub$ is chosen from the distribution $\pi$. Here $1/|\PP_u|=0$ if
$\PP_u$ is an infinite set. Clearly $0\le\psi(\PP)\le1$. Let us note that if
$J$ is finite and $\pi$ is the uniform distribution on $J$, then
\begin{equation}\label{EQ:FIN-P}
\psi(\PP) = \E\Big(\frac1{|\PP_\ub|}\Big) = \sum_{A\in\PP} \frac{|A|}{|J|}\,\frac1{|A|} = \frac{|\PP|}{|J|}.
\end{equation}
It is also clear that $\psi$ is a decreasing function on partitions (it is
smaller on coarser partitions).

Our main lemma is the following.

\begin{lemma}\label{LEM:PSI-SUPERMOD}
Let $(J,\BB,\pi)$ be a standard Borel probability space and let $\psi$ be
defined by \eqref{EQ:PSI}. Let $\PP$, $\QQ$ and $\RR$ be $\BB$-partitions with
the re-randomizing property such that $\RR\le\PP$ and $\RR\le\QQ$. Then
\begin{equation}\label{EQ:MAIN-SUPER}
\psi(\RR) + \psi(\PP\lor\QQ) \ge \psi(\PP) + \psi(\QQ).
\end{equation}
\end{lemma}

We see that $\RR$ is a ``proxy'' for $\PP\land\QQ$, which may not be a
$\BB$-partition. In the next section, when applying the lemma to graphons, we
can choose $\RR$ to be a natural meet of the partitions $\PP$ and $\QQ$ (which
will not be $\PP\land\QQ$).

\begin{proof}
$0^\circ$\ Let us prove the finite case first, when $\BB=2^J$. We don't have to
worry about $\PP\land\QQ$ being a $\BB$-partition, so we can replace $\RR$ by
$\PP\land\QQ$. Let $K_n$ be the complete graph on $J$. We consider a partition
$\PP$ as a subgraph $G_\PP$ of $K_n$ that is the disjoint union of complete
graphs, corresponding to the partition classes. Then $G_{\PP\land\QQ}=G_\PP\cap
G_\QQ$ and $G_{\PP\lor\QQ}$ is the transitive closure of $G_\PP\cup G_\QQ$.
Furthermore, $n(1-\psi(\PP))$ is the rank of $G_\PP$ in the cycle matroid of
$K_n$. The inequality follows by the submodularity of the rank function of the
cycle matroid.

\medskip

$1^\circ$\ As a first step in the proof of the general case, we want to split
the classes of $\QQ_\fin$ into some subsets, and those classes in $\RR$
contained in a class in $\QQ_\fin$ to singletons, so that the inequality to be
proved remains equivalent.

Let $S$ be a finite-class representative for $\RR$ (which exists by Lemma
\ref{LEM:REPRESENTATIVE}). We split each $Y\in\QQ_\fin$ into $Y\cap S$ and
singleton sets, and let $\QQ'$ be the partition of $J$ obtained. Let $\RR'$ be
obtained by splitting every class of $\RR$ contained in $\cup\QQ_\fin$ to
singletons. It is easy to see that $\PP\lor\QQ'=\PP\lor\QQ$.

Let $\ub$ be a random point from $\pi$, and let $\vb$ be obtained by
re-randomizing $\ub$ along $\QQ$. If $\QQ_\ub$ is infinite, then no partition
class containing $\vb$ changes when replacing $\QQ$ by $\QQ'$. If $Y=\QQ_\ub$
is finite, then
\[
\frac1{|(\PP\lor\QQ')_\vb|} = \frac1{|(\PP\lor\QQ)_\vb|}, \et \frac1{|\RR'_\vb|} = 1.
\]
Using \eqref{EQ:FIN-P}, we have
\[
\E_\vb\Big(\frac1{|\QQ'_\vb|}\Big) = \frac{|Y\setminus S|+1}{|Y|}, \et \E_\vb\Big(\frac1{|\QQ_\vb|}\Big)= \frac{1}{|Y|}.
\]
Furthermore,
\[
\E_\vb\Big(\frac1{|\RR_\vb|}\Big) = \frac{\big|\RR|_Y\big|}{|Y|} = \frac{|Y\cap S|}{|Y|}.
\]
Combining these equations, we see that
\begin{align*}
\E_\vb\Big(&\frac1{|(\PP\lor\QQ')_\vb|}\Big) + \E_\vb\Big(\frac1{|\RR'_\vb|}\Big)
-\E_\vb\Big(\frac1{|\PP_\vb|}\Big) -\E_\vb\Big(\frac1{|\QQ'_\vb|}\Big)\\
&=\E_\vb\Big(\frac1{|(\PP\lor\QQ)_\vb|}\Big) + 1-\E_\vb\Big(\frac1{|\PP_\vb|}\Big) - \frac{|Y\setminus S|+1}{|Y|}\\
&=\E_\vb\Big(\frac1{|(\PP\lor\QQ)_\vb|}\Big) + \frac{|Y\cap S|}{|Y|}-\E_\vb\Big(\frac1{|\PP_\vb|}\Big) - \frac1{|Y|}\\
&= \E_\vb\Big(\frac1{|(\PP\lor\QQ)_\vb|}\Big) + \E_\vb\Big(\frac1{|\RR_\vb|}\Big)
-\E_\vb\Big(\frac1{|\PP_\vb|}\Big) -\E_\vb\Big(\frac1{|\QQ_\vb|}\Big).
\end{align*}
By Lemma \ref{LEM:RRAND-INT}, this implies that this equation holds when $\vb$
is replaced by $\ub$, and so
\begin{align*}
\psi(\RR') &+ \psi(\PP\lor\QQ')-\psi(\PP) - \psi(\QQ')
=\psi(\RR) + \psi(\PP\lor\QQ)-\psi(\PP) - \psi(\QQ),
\end{align*}
so considering the pair $(\PP,\QQ')$ instead of $(\PP,\QQ)$, we may assume that
every class of $\RR$ is a singleton.

\medskip

$2^\circ$\ Our next goal is to get rid of the non-singleton finite classes of
$\PP\lor\QQ$. Let $Z$ be the union of all finite classes of $\PP\lor\QQ$, and
let $\PP'$ and $\QQ'$ be obtained by splitting each class of $\PP$ and $\QQ$
contained in $Z$ to singletons. It is easy to see that $Z\in\BB$, and that both
$\PP'$ and $\QQ'$ are $\BB$-partitions. By Lemma \ref{LEM:REFINE-RESAMPLE},
both $\PP'$ and $\QQ'$ have the re-randomizing property.

We can write the inequality to be proved as
\begin{equation}\label{EQ:PUQU}
\intl_J f_{\PP,\QQ}(x)\,d\pi(x)\ge 0,
\end{equation}
where
\[
f(x)=f_{\PP,\QQ}(x)=\frac1{|\RR_x|}+\frac1{|(\PP\lor\QQ)_x|}-\frac1{|\PP_x|}-\frac1{|\QQ_x|}.
\]
Let $f'=f_{\PP',\QQ'}$. Then $f'(x)=f(x)$ if $x\in J\setminus Z$, and $f'(x)=0$
if $x\in Z$. By the finite case treated above, $\intl_Z f(x)\,d\pi(x)\ge 0$,
and hence
\[
\intl_J f(x)\,d\pi(x)\ge \intl_J f(x)\,d\pi(x).
\]
So we may replace $(\PP,\QQ)$ by $(\PP',\QQ')$; in other words, we may assume
that every class of $\PP\lor\QQ$ is either infinite or a singleton set.

\medskip

$3^\circ$\ Our next goal is to get rid of non-singleton finite classes
containing a singleton class from the other partition. Let $S$ and $T$ be the
unions of singleton classes of $\PP$ and $\QQ$, respectively. For $Y\in\QQ$,
define
\[
Y'=
  \begin{cases}
    Y\setminus S, & \text{if $Y$ is finite}, \\
    Y, & \text{otherwise}.
  \end{cases}
\]
If $|Y|>1$, then the class of $\PP\lor\QQ$ is infinite, so $Y'\not=\emptyset$.
Let the partition $\QQ'$ consist of the nonempty sets $Y'$ ($Y\in\QQ$), along
with those singleton subsets of $S$ that are contained in finite classes of
$\QQ$. We leave $\RR$ unchanged.

We argue similarly as in $1^\circ$. Let $\ub$ be a random point from $\pi$, and
let $\vb$ be obtained by re-randomizing $\ub$ along $\QQ$. Fix $\ub$, and let
$Y=\QQ_\ub$. Then If $Y$ is infinite, then no partition class containing $\vb$
changes when replacing $\QQ$ by $\QQ'$. If $Y$ is finite, then by $1^\circ$,
for any point $\vb\in Y$,
\[
\frac1{|\RR_\vb|} =1,\et \frac1{|\QQ_\vb|} =\frac1{|Y|},
\]
and by $2^\circ$,
\[
\E_\vb\Big(\frac1{|(\PP\lor\QQ)_\vb|}\Big) = 0.
\]
Furthermore, by \eqref{EQ:FIN-P},
\[
\E_\vb\Big(\frac1{|\QQ'_\vb|}\Big) = \frac{|Y\setminus Y'|+1}{|Y|},
\]
and
\[
\E_\vb\Big(\frac1{|(\PP\lor\QQ')_\vb|}\Big) = \frac{|Y\setminus Y'|}{|Y|}.
\]
Combining these equations,
\begin{align*}
\E_\vb\Big(&\frac1{|(\PP\lor\QQ')_\vb|}\Big) + \E_\vb\Big(\frac1{|\RR_\vb|}\Big)
-\E_\vb\Big(\frac1{|\PP_\vb|}\Big) -\E_\vb\Big(\frac1{|\QQ'_\vb|}\Big)\\
&=\frac{|Y\setminus Y'|}{|Y|} + 1-\E_\vb\Big(\frac1{|\PP_\vb|}\Big) - \frac{|Y\setminus Y'|+1}{|Y|}
=0+1- \E_\vb\Big(\frac1{|\PP_\vb|}\Big) - \frac1{|Y|}\\
&= \E_\vb\Big(\frac1{|(\PP\lor\QQ)_\vb|}\Big) + \E_\vb\Big(\frac1{|\RR_\vb|}\Big)
-\E_\vb\Big(\frac1{|\PP_\vb|}\Big) -\E_\vb\Big(\frac1{|\QQ_\vb|}\Big).
\end{align*}
As before, this shows that
\begin{align*}
\psi(\PP\land\QQ') &+ \psi(\PP\lor\QQ')-\psi(\PP) - \psi(\QQ')\\
&=\psi(\PP\land\QQ) + \psi(\PP\lor\QQ)-\psi(\PP) - \psi(\QQ).
\end{align*}
Carrying out the same process with $\PP$ and $\QQ$ interchanged, we may assume
that for every point, three possibilities remain: it is contained (i) in
singleton classes of both partitions $\PP$ and $\QQ$, or (ii) in a singleton
class and an infinite class, or (iii) in two non-singleton classes.

In these three cases we have, respectively,
\begin{alignat*}{2}
&\text{(i):}~&f_{\PP,\QQ}(x) &= 1+1-1-1=0,\\
&\text{(ii):}~&f_{\PP,\QQ}(x) &= 1+0-1-0=0,\\
&\text{(iii):}~&f_{\PP,\QQ}(x) &\ge 1+0-\frac12-\frac12=0.
\end{alignat*}
This proves the Lemma.
\end{proof}

\section{Cycle matroids of graphings}\label{SSEC:GRAPHING-MAT}

\subsection{Graphing basics}

A {\it graphing} is a quadruple $\Gb=(J,\BB,E,\lambda)$, where $(J,E)$ is a
graph on the standard Borel sigma-algebra $(J,\BB)$, $E$ is a symmetric Borel
subset of $\binom{J}{2}$, and $\lambda$ is a probability measure on $\BB$. Let
$\deg(u)=\deg_\Gb(u)$ denote the degree of $u\in J$ in the graph $\Gb$; for
$B\subseteq J$, let $\deg_B(u)$ denote the number of edges from the point $u\in
J$ to $B$; for $X\subseteq E$, let $\deg_X(u)$ denote the number of edges in
$X$ incident with $u$. For $x\in J$, we denote by $\Gb_x$ the connected
component of $\Gb$ containing $x$.

It is assumed that all degrees of the graph $(J,E)$ are finite and bounded by
an integer $D\ge 0$, and the ``measure-preservation'' equation
\begin{equation}\label{EQ:AB-BA}
\intl_A\deg_B(x)\,d\lambda(x)=\intl_B\deg_A(x)\,d\lambda(x)
\end{equation}
is satisfied for all $A,B\in\BB$. This condition is equivalent to saying that
for a random point $\xb$ from distribution $\lambda$, the (connected rooted)
graph $(G_\xb,\xb)$ is an involution-invariant random rooted graph as defined
by Benjamini and Schramm \cite{BSch}.

The quantity
\[
\overline{d} = \overline{d}_\Gb = \intl_J\deg(u)\,d\lambda(u)
\]
is the average degree of the graphing. The setfunction
\[
\eta(A\times B)=\frac1{\overline{d}} \intl_A\deg_B(x)\,d\lambda(x)
\]
extends to a probability measure on $\BB\times\BB$, which we call the {\it edge
measure} of $\Gb$. Clearly $\eta=\eta_\Gb$ is symmetric (invariant under
interchanging the coordinates) and supported on $E$. We can express $\eta(X)$
for every Borel set $X\subseteq E$ as
\begin{equation}\label{EQ:ETA-PI2}
\eta(X) = \frac1{\overline{d}} \intl_J \deg_X(u)\,d\lambda(u).
\end{equation}

In this paper, a {\it subgraphing} of a graphing $\Gb=(J,\BB,E,\lambda)$ is a
$4$-tuple $\Hb=(J,\BB,F,\lambda)$, where $F\subseteq E$ is a Borel set. It is
not entirely obvious that $\Hb$ satisfies the graphing axioms; see Lemma 18.19
in \cite{HomBook}. Equation \eqref{EQ:ETA-PI2} implies the following expression
for the edge measure of a subgraphing:
\begin{equation}\label{EQ:EDGE-SUB}
\eta_\Hb(X)=\frac{\overline{d}}{\overline{d}_\Hb}\eta(X)\qquad (X\in\BB\times\BB).
\end{equation}

\subsection{The rank function}

An obvious way to define a matroid rank function of a graphing would be to
consider finite subsets of edges, and consider the rank function $r(X)$ of the
cycle matroid of the graph they form. Every infinite set will have infinite
rank. We explore a more interesting possibility, generalizing the {\it
normalized rank} of the cycle matroid of a graph $G=(V,E)$ on $n$ nodes.

Let $c(G)$ denote the number of connected components of a graph $G$. Recalling
that $r(X)=n-c(V,X)$, we define
\[
\rho(X)=\frac{r(X)}{n}= 1-\frac{c(V,X)}{|V|}.
\]
This number is in $[0,1]$. By \eqref{EQ:FIN-P}, we can express it as
\begin{equation}\label{EQ:RHO-EXPECT}
\rho(X) = 1-\E_\ub\Big(\frac1{|X_\ub|}\Big),
\end{equation}
where $X_\ub$ is the set of nodes of the connected component of $(V,X)$
containing $\ub$.

Now this setfunction makes sense in the graphing case as well. Let
$\Gb=(J,\BB,E,\lambda)$ be a graphing. For every Borel set $X\subseteq E$, the
quadruple $\Gb^X=(J,\BB,X,\lambda)$ is a graphing. We denote by $\PP^X$ the
partition of $J$ into the connected components of $\Gb^X$. Then we can define
\begin{equation}\label{EQ:PSI-GRAPHING}
\rho(X)=\rho_\Gb(X)=1-\E_\ub\Big(\frac1{|V(\Gb^X_\ub|}\Big)=1-\psi(\PP^X),
\end{equation}
where $\ub$ is a random point from $\lambda$. We can see from this definition
right away a nice property of $\rho(X)$, namely that it is an intrinsic
parameter of the edge set $X$: if $\Hb$ is a subgraphing of $\Gb$, and $X$ is a
Borel subset of $E(\Hb)$, then $\rho_\Hb(X)=\rho_\Gb(X)$.

Our goal is to apply the results of Section \ref{SSEC:PART-SUBMOD} to the
partitions of $J$ into the connected components of sub-graphings, showing that
$1-\rho$ is supermodular, and so $\rho$ is submodular. For this, we need a
simple (but not quite trivial) lemma.

\begin{lemma}\label{LEM:RESAMPLE}
Let $\Gb=(J,\BB,E,\lambda)$ be a graphing, and let $\PP$ be the partition of
$J$ into the connected components of $\Gb$. Then $\PP$ is a $\BB$-partition
with the re-randomizing property.
\end{lemma}

\begin{proof}
The fact that $\PP$ is a $\BB$-partition, follows e.g.\ by Lemma 18.2 in
\cite{HomBook}, which asserts that if $A\in\BB$, then the neighbors of $A$ form
a Borel set. Repeating this argument, the set $A_k$ of nodes that are at a
distance at most $k$ form $A$ is a Borel set, and hence $\PP(A)=\cup_kA_k$ is
Borel.

To prove that $\PP$ has the re-randomizing property, let $\ub$ be a random
point from $\lambda$, let $\vb$ be obtained from $\ub$ by re-randomizing along
$\PP$, and let $\lambda'$ be the distribution of $\vb$. Let $B=\cup\PP_\fin$.
Since $\lambda'|_{J\setminus B}=\lambda|_{J\setminus B}$ by definition, it
suffices to show that $\lambda'(A)=\lambda(A)$ for all $A\in\BB$, $A\subseteq
B$. For $x,y\in J$, define
\[
f(x,y)=
  \begin{cases}
    1/|V(\Gb_x)|, & \text{if $x\in B$ and $y\in A\cap V(G_x)$,}\\
    0, & \text{otherwise}.
  \end{cases}
\]
Then
\[
\sum_y f(x,y) =
  \begin{cases}
    \displaystyle\frac{|V(\Gb_x)\cap A|}{|V(\Gb_x)|}, & \text{if $x\in B$}, \\
    0, & \text{otherwise},
  \end{cases}
\]
and
\[
\sum_x f(x,y) =
  \begin{cases}
    1, & \text{if $y\in A$}, \\
    0, & \text{otherwise}.
  \end{cases}
\]
By the Mass Transport Principle (in the form of \cite{HomBook}, Proposition
18.49),
\begin{align*}
\lambda(A) &= \intl_J \sum_x f(x,y)\,d\lambda(y) = \intl_J \sum_y f(x,y)\,d\lambda(x)\\
&=\intl_B \frac{|V(\Gb_x)\cap A|}{|V(\Gb_x)|}\,d\lambda(x) = \Pr(u\in B, u'\in A) = \lambda'(A).
\end{align*}
This proves the re-randomizing property of $\PP$.
\end{proof}

Another simple fact we need is the following.

\begin{lemma}\label{LEM:RHO-ETA}
The setfunction $\rho=\rho_\Gb$ defined by \eqref{EQ:PSI-GRAPHING} satisfies
\[
\frac{\overline{d}}{1+D} \eta(X)\le\rho(X)\le \overline{d}\eta(X).
\]
for all $X\in\BB\times\BB$.
\end{lemma}

\begin{proof}
Clearly $|V(\Gb^X_u)|\ge 1+\deg_X(u)$, and hence
\[
1-\frac1{|V(\Gb^X_u)|} \ge 1-\frac1{1+\deg(u,X)} = \frac{\deg(u,X)}{1+\deg(u,X)}\ge
\frac{\deg(u,X)}{1+D}.
\]
On the other hand,
\[
1-\frac1{|V(\Gb^X_u)|}\le \deg(u,X),
\]
since both sides are zero if $\deg(u,X)=0$, and the left hand side is at most
$1$ and the right hand side is at least $1$ otherwise. Taking expectation, and
using \eqref{EQ:ETA-PI2}, we get the inequalities in the lemma.
\end{proof}

Let us state some nice properties of the setfunction $\rho$ (see \cite{Lov23a}
for more on these properties and their relevance). We say that a setfunction
$\fg:~\BB\to\R_+$ is {\it strongly bounded}, if there is a (finite) measure
$\alpha$ such that $\fg\le\alpha$. It is {\it continuous from above}, if for
every sequence $A_1\supseteq A_2\supseteq\ldots$ of sets $A_i\in\BB$ and
$A=\cap_i A_i$, we have $\fg(A_n)\to \fg(A)$ $(n\to\infty)$. It is {\it
absolutely continuous with respect to a measure $\pi$}, if for every $\eps>0$
there is a $\delta>0$ such that $\pi(S)<\delta$ implies that $|\fg(S)|<\eps$
for every $S\in\BB$.

\begin{theorem}\label{COR:PSI-SUPERMOD}
The setfunction $\rho$ defined on the Borel subsets of $E$ by
\eqref{EQ:PSI-GRAPHING} is increasing and submodular. In addition, it is
strongly bounded, continuous from above and absolutely continuous with respect
to a measure $\eta$.
\end{theorem}

\begin{proof}
To prove submodularity, we use Lemma \ref{LEM:PSI-SUPERMOD}. Equivalently, we
want to prove that $\psi(\PP^X)$ is a supermodular setfunction of $X$:
\begin{equation}\label{EQ:GR-SUPER}
\psi(\PP^{X\cup Y}) + \psi(\PP^{X\cap Y}) \ge \psi(\PP^X) + \psi(\PP^Y).
\end{equation}
Here $\PP^{X\cup Y}$ etc.\ are $\BB$-partitions with the re-randomizing
property by Lemma \ref{LEM:RESAMPLE}. Furthermore, $\PP^{X\cup Y}=\PP^X\lor
\PP^Y$. In general, we will not have $\PP^{X\cap Y}=\PP^X\land \PP^Y$, but
$\RR=\PP^{X\cap Y}$ satisfies $\RR\le \PP^X$ and $\RR\le \PP^Y$, so Lemma
\ref{LEM:PSI-SUPERMOD} proves \eqref{EQ:GR-SUPER}.

The smoothness properties are trivial by Lemma \ref{LEM:RHO-ETA}.
\end{proof}

\subsection{Bases and minorizing measures}

Bases of the cycle matroid of a graph are spanning forests, so we expect that
bases of the matroid of a graphing (if meaningful at all) will be connected to
spanning trees. It is suggested in \cite{Lov23a}, that the ``bases'' for a
general increasing submodular setfunction $\fg$ with $\fg(\emptyset)=0$ should
be {\it minorizing charges}, i.e., nonnegative finitely additive measures
$\alpha$ on $\BB$ such that $\alpha\le\fg$.

It was proved (under somewhat different conditions) by Choquet \cite{Choq} and
\v{S}ipo\v{s} \cite{Sipos} (see also \cite{Denn}, Chapter 10), that for every
increasing submodular setfunction $\fg$ defined on Borel sets with
$\fg(\emptyset)=0$ and every family $\SS$ of Borel sets that is totally ordered
by inclusion, there is always a minorizing charge $\alpha$ with
$\alpha(J)=\fg(J)$. In our case, Theorem \ref{COR:PSI-SUPERMOD} implies that
every minorizing charge for $\rho$ is in fact countably additive. See
\cite{Lov23a} for other matroid-like properties of minorizing measures, like
Steinitz exchange, greedy algorithm, intersection, etc.

So we expect that spanning forests give rise to minorizing measures. In the
finite case, if $X\subseteq E$ is a spanning forest, then $\rho(Y)=|Y|/|V|$ for
every $Y\subseteq X$, and so $\rho$ is additive on $X$. This does not hold for
graphings in general (see Example \ref{EXA:TREE-NONADD}), but we can prove this
at least in a special case. We say that a graphing $\Gb=(J,\BB,E,\lambda)$ is
{\it hyperfinite}, if for every $\eps>0$ there is a Borel set $X\subseteq E$
with $\eta(X)<\eps$ such that all connected components of $\Gb\setminus X$ are
finite.

\begin{lemma}\label{LEM:FORESTS}
Let $\Gb=(J,\BB,E,\lambda)$ be a hyperfinite graphing such that every connected
component is a tree. Then for all Borel sets $U\subseteq E$, we have
$\rho(U)=(\overline{d}/2)\eta(U)$.
\end{lemma}

In particular, $\rho$ is a measure. Applying this lemma with $U=E$, we obtain
that $\rho(E)=\overline{d}/2$, so the lemma could be phrased as
$\rho(U)=\rho(E)\eta(U)$.

\begin{proof}
First, assume that every connected component $Q$ of $\Gb$ is a finite tree.
Then $U\cap E(Q)$ is independent in the cycle matroid of $Q$, and hence
\[
\rho_Q(U\cap E(Q))= \frac{|U\cap E(Q)|}{|V(Q)|}.
\]
The quantities on both side can be expressed using a random point $\qb$ of
$V(Q)$.
\[
\rho_Q(U\cap E(Q)) = \E_\qb\Big(1-\frac1{|U_\qb|}\Big),
\]
and
\[
|U\cap E(Q)| =\frac12 \sum_{q\in V(Q)} \deg_U(q) = \frac{|V(Q)|}2 \E_\qb \deg_U(\qb).
\]
Thus
\[
\E_\qb\Big(1-\frac1{|U_\qb|} - \frac12 \deg_U(\qb)\Big) =0.
\]
Since this holds for every connected component $Q$ of $\Gb$, Lemma
\ref{LEM:RRAND-INT} implies that the same relation holds for a random point
$\xb$ from $\lambda$:
\[
\E_\xb\Big(1-\frac1{|U_\xb|} - \frac12 \deg_U(\xb)\Big) =0.
\]
But here
\[
\E_\xb\Big(1-\frac1{|U_\xb|}\Big) = \rho(U) \et \frac12\E_\xb \big(\deg_U(\xb)\big)
= \frac{\overline{d}}{2}\eta(U)
\]
This settles the case when all components of $\Gb$ are finite.

Now let $\Gb$ be any hyperfinite forest. Fix an $\eps>0$. By hyperfiniteness,
there is a Borel set $X\subseteq E$ such that $\eta(X)<\eps$ and every
connected component of $\Hb=\Gb\setminus X$ is finite. By Lemma
\ref{LEM:RHO-ETA} we have $\rho(U\cap X)\le D\eta(U\cap X) \le D\eps$, and so
\[
\rho(U\setminus X)\le \rho(U)\le \rho(U\setminus X) + \rho(U\cap X)\le \rho(U\setminus X) + D\eps.
\]
As noted before, we have $\rho(U\setminus X)=\rho_\Hb(U\setminus X)$. Hence by
the special case proved above and by \eqref{EQ:EDGE-SUB},
\[
\rho(U\setminus X)=\frac{\overline{d}_\Hb}{2}\eta_\Hb(U\setminus X) = \frac{\overline{d}}{2}\eta(U\setminus X),
\]
and so
\[
\rho(U) \le \rho(U\setminus X) + D\eps = \frac{\overline{d}}{2}\eta(U\setminus X) + D\eps \le \frac{\overline{d}}{2}\eta(U)+D\eps.
\]
while
\[
\rho(U) \ge \rho(U\setminus X) = \frac{\overline{d}}{2}\eta(U\setminus X) \ge
\frac{\overline{d}}{2}\eta(U)-\eps.
\]
Letting $\eps\to0$, the lemma follows.
\end{proof}

We can describe a class  of minorizing measures for $\rho$. In the finite case,
these are all the extremal minorizing measures. In the infinite case there must
be others, but constructing all of them will probably be very difficult.

\begin{corollary}\label{COR:FORESTS}
Let $\Gb$ be a graphing and $\Hb=(J,F)$, a hyperfinite spanning subforest of
$\Gb$. Then
\[
\alpha(X) = \frac{\overline{d}}{2}\eta(X\cap F)
\]
defines a measure on the Borel subsets of $E$ such that $\alpha\le\rho$ and
$\alpha(E)=\rho(E)$.
\end{corollary}

\begin{proof}
It is clear that $\alpha$ defines a measure. Using \eqref{EQ:EDGE-SUB} and
Lemma \ref{LEM:FORESTS}, we get
\[
\alpha(X)=\frac{\overline{d}}{2}\eta(X\cap F)
= \frac{\overline{d}}{2}\, \frac{\overline{d}_\Hb}{\overline{d}}\, \eta_\Hb(X\cap F)
= \frac{\overline{d}_\Hb}{2} \frac2{\overline{d}_\Hb} \rho(X\cap F),
\]
and hence $\alpha(X)\le \rho(X)$, and $\alpha(E)=\alpha(F)=\rho(F)=\rho(E)$.
\end{proof}

\begin{example}\label{EXA:TREE-NONADD}
We cannot drop the hyperfiniteness assumption in Lemma \ref{LEM:FORESTS}.
Indeed, let $D=2r-1$ be an odd integer $(r\ge 3)$, and let $\Gb$ be a graphing
such that every connected component is a $D$-regular tree (such graphings
exist, for example the Bernoulli graphing of a $D$-regular rooted tree, see
e.g. \cite{HomBook}). By a theorem of Csóka, Lippner and Pilhurko \cite{CLP},
the edge set of $\Gb$ has a partition $E_0\cup E_1\cup\dots\cup E_{D+1}$ such
that $E_i\in\BB\times\BB$, $E_1,\dots,E_{D+1}$ are matchings, and $E_0$ covers
a set $S_0$ of measure $0$. Let $U=E_1\cup\dots\cup E_r$ and $W=E\setminus U$.
Then every $u\in J\setminus S_0$ misses at most one of $E_1,\dots,E_{D+1}$.
Hence the connected component of $(J,U)$ containing $u$ has at least $r$ nodes,
and the same holds for $(J,W)$. This implies that
\[
\rho(U)+\rho(W) \ge 2\big(1-\frac1r\big) = 2-\frac2r >1 = \rho(U\cup W).
\]
So $\rho$ is not even additive.
\end{example}

\subsection{Beginnings of a limit theory}

A next step should be to develop the limit theory for $\rho$; in other words,
to show that if a graph sequence $(G_n:~n=1,2,\dots)$ with uniformly bounded
degrees locally converges to a graphing $\Gb$, then the setfunctions
$\rho_{G_n}$ converge to the setfunction $\rho_\Gb$ in some sense. This is not
resolved in this paper; but we do prove that the total rank converges. Given a
graphing $\Gb=(J,\BB,E,\lambda)$, we define its {\it total rank} as
\begin{equation}\label{EQ:RK-DEF}
\rk(\Gb)=\rho_\Gb(E(\Gb)) = 1-\E\Big(\frac1{|V(\Gb_\ub)|}\Big),
\end{equation}
where $\ub$ is a random point from $\lambda$. The total rank of a finite graph
is defined analogously.

A graphing parameter $f$ is called {\it locally estimable}, if given an error
bound $\eps>0$ and degree bound $D\ge 0$, there are integers $k,r\ge 1$
depending only on $\eps$ and $D$, such that for any graphing
$\Gb=(J,\BB,E,\lambda)$ with maximum degree $D$, selecting $k$ independent
random nodes $x_1,\ldots,x_k$ from $\lambda$ and considering their
neighborhoods $B_r(x_1),\ldots, B_r(x_k)$, we can compute an estimate $R$ of
$f(G)$ such that
\begin{equation}\label{EQ:RANK-EST}
\Pr\big(|R - f(G)|\ge\eps\big) \le \eps.
\end{equation}
This definition applies, in particular, to estimating $f$ on finite graphs.

\begin{theorem}\label{THM:LOC-EST}
{\rm(a)} The graphing parameter $\rk$ is locally estimable. {\rm(b)} If
$G_1,G_2,\ldots$ is a sequence of finite graphs with all degrees bounded by
$D\ge 0$ locally converging to a graphing $\Gb$, then $\rk(G_n) \to \rk(\Gb)$
as ($n\to\infty$).
\end{theorem}

\begin{proof}
Choose $k=\lceil 2/\eps\rceil$. Then for every $x\in J$ we have
$B_k(x)\subseteq V(x)$ and
\[
|B_k(x)|
  \begin{cases}
    = |V(x)|, & \text{if $B_k(x)=V(x)$}, \\
    \ge k, & \text{otherwise}.
  \end{cases}
\]
Hence
\[
\frac1{|V(x)|} \le \frac1{|B_k(x)|}\le \frac1{|V(x)|} + \frac1k.
\]
Taking expectation in $x$, we get
\[
\rk(\Gb) \ge 1-\E_x\Big(\frac1{|B_k(x)|}\Big) \ge \rk(\Gb) - \frac\eps2.
\]
Repeating this $N=O(\eps^{-2}\log\eps^{-1})$ times independently, the quantity
\[
R=1-\frac1N \sum_{i=1}^N\frac1{|B_k(x_i)|}
\]
will be closer to $\rk(\Gb)$ than $\eps$ with probability at least $1-\eps$.
The second assertion follows easily.
\end{proof}

\begin{remark}\label{REM:GRAPHING}
The total rank is even easier to define for the Benjamini--Schramm
representation of the weak limit \cite{BSch}, namely an involution-invariant
random rooted connected graph $(G,o)$. Then
\[
\overline{\rho}(\Gb)=1-\E\Big(\frac1{|V(G,o)|}\Big).
\]
However, the submodularity over subgraphs would be more difficult to formulate.
\end{remark}

\end{document}